\documentclass[12pt]{article}
\usepackage{graphicx} 
\usepackage{caption} 
\usepackage{epstopdf}
\usepackage{subfigure}
\usepackage{color}

\usepackage[english]{babel}
\usepackage{amsmath,amsthm}
\usepackage{amsfonts}
\usepackage{booktabs}

\usepackage{indentfirst} 
\usepackage{graphicx}
\usepackage{subfigure}
\usepackage{tabularx}
\newtheorem{thm}{Theorem}[section]
\newtheorem{lem}{Lemma}[section]
\newtheorem{cor}{Corollary}[section]
\newtheorem{rem}{Remark}[section]

\newtheorem{defi}{Definition}

\title{The matching book embedding of the $F${-}sum of two graphs}

\author { Zeling Shao, Ruxing Sun, Zhiguo Li{$^*$}\\
	{\small School of Science, Hebei University of Technology, Tianjin 300401, China}
	\date{}
	\footnote{Corresponding author. E-mail: zhiguolee@hebut.edu.cn}
}

\begin{document}
	\baselineskip 0.65cm
	
	\maketitle
	
	\begin{abstract}
		
		The $F$-sum is a new graph operation defined by combining four graph transformation operations with the Cartesian product operation. 
		A \emph{matching book embedding} of a graph $G$ is a book embedding in which the vertices of $G$ are placed on a fixed linear order along the spine, and the edges are assigned to pages such that 
		(i) no two edges on the same page cross, and 
		(ii) each vertex has degree at most one on every page. 
		The minimum number of pages required for such a matching book embedding is called the \emph{matching book thickness} of $G$, denoted by $mbt(G)$. 
		A graph \( G \) is \emph{dispersable} if and only if \( mbt(G) = \Delta(G) \), and \emph{nearly dispersable} if and only if \( mbt(G) = \Delta(G) + 1 \). 
		In this paper, we determine the dispersability of outerplanar graphs and establish an upper bound on the matching book thickness of the $F$-sum of any simple graph with any dispersable bipartite graph.

		
		\bigskip
		\noindent\textbf{2020 MR Subject Classification. 05C10}
		
	\end{abstract}

	\section{Introduction}
	
	Bernhart and Kainen$^{[1]}$ first introduced the concept of book embedding of graphs. 
	A \emph{book embedding} of a graph $G$ consists of placing all vertices of $G$ in a fixed linear order along a straight line in three-dimensional space(called the \emph{spine}) and assigning each edge to a half-plane (called \emph{page}) bounded by the spine, such that no two edges assigned to the same page cross. 
	If, in addition, every vertex has degree at most one in each page, the method is called a \emph{matching book embedding}. 
	The minimum number of pages required for such a matching book embedding is known as the \emph{matching book thickness} of $G$, denoted by $\operatorname{mbt}(G)$.
	A graph $G$ is said to be \emph{dispersable} if $mbt(G) = \Delta(G)$, and \emph{nearly dispersable} if $mbt(G) = \Delta(G) + 1$.
	Equivalently, a graph $ G $ with edge chromatic number $ k $ and maximum vertex degree $ \Delta(G) = k $ is \emph{dispersable} if and only if all vertices of $ G $ can be placed on a cycle (called \emph{printing cycle}) and its edges are represented as chords of the cycle, such that $ G $ admits a $ k $-page matching book embedding where all edges on each page share the same color. Overbay studied the dispersability of complete graphs, complete bipartite graphs, trees, even cycles, and $ n $-cubes$^{[2]}$. Currently, the matching book embedding problems for various graph families have been extensively studied$^{[6-11]}$.
	
	In 2008, Eliasi and Taeri$^{[3]}$ introduced four new graph operations. 
	Based on these operations and the Cartesian product, they proposed the $F$-sum of two graphs $G_1$ and $G_2$, whose vertex set is defined as $(E(G_1) \cup V(G_1)) \square V(G_2)$.
	
	In this paper, we determine the dispersability of outerplanar graphs and establish an upper bound on the matching book thickness of the $F$-sum of any simple graph and any dispersable bipartite graph.

	\section{Preliminaries}
	
	In this section, we present some definitions and results that will be used in the following sections.

	\begin{defi}$^{[3]}$
		For a connected graph $G$, define four related graphs as follows (see Fig. 1):
		
		$(a)$ $S(G)$ is the graph obtained by inserting an additional vertex in each edge of $G$. Equivalently, each edge of $G$ is replaced by a path of length 2.
		
		$(b)$ $R(G)$ is obtained from $G$ by adding a new vertex corresponding to each edge of $G$, then joining each new vertex to the end vertices of the corresponding edge.
		
		$(c)$ $Q(G)$ is obtained from $G$ by inserting a new vertex into each edge of $G$, then joining with edges those pairs of new vertices on adjacent edges of $G$.
		
		$(d)$ $T(G)$ has as its vertices the edges and vertices of $G$. Adjacency in $T(G)$ is defined as adjacency or incidence for the corresponding elements of $G$. (see Fig.~1 for $S_3$ and $F(S_3)$)
	\end{defi}

	\begin{figure}[htbp]
		\centering
		\includegraphics[height=2.8cm, width=0.9\textwidth]{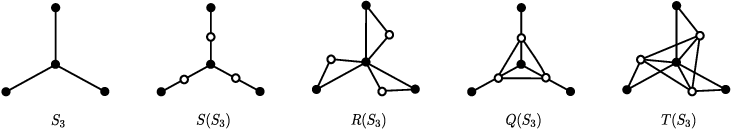}
		
		\centerline{Fig.~1 $S_3, S(S_3), R(S_3), Q(S_3), T(S_3)$}
	\end{figure}

	\begin{defi}$^{[4]}$
		Let $F \in \{S, R, Q, T\}$. The $F$-sum of $G_1$ and $G_2$, denoted by $G_1 {+_F} G_2$, is defined by $F(G_1) \square G_2 - E^*$, where $E^* = \{(u, v_1)(u, v_2) \in E(F(G_1) \square G_2) : u \in V(F(G_1)) - V(G_1), v_1v_2 \in E(G_2)\}$, i.e., $G_1 +_F G_2$ is a graph with the set of vertices $V(G_1 +_F G_2) = (V(G_1) \cup E(G_1)) \square V(G_2)$ and two vertices $(u_1, u_2)$ and $(v_1, v_2)$ of $G_1 +_F G_2$ are adjacent if and only if $[u_1 = v_1 \in V(G_1) \text{ and } u_2v_2 \in E(G_2)]$ or $[v_2 = v_2 \in V(G_2) \text{ and } u_1v_1 \in E(F(G_1))]$.
	\end{defi}

	Observe that $ G_1 +_F G_2 $ consists of $ |V(G_2)| $ copies of the graph $ F(G_1) $, each associated with a vertex of $ G_2 $. In each copy, the vertices are of two types: those corresponding to the vertices of $ G_1 $, referred to as black vertices, and those corresponding to the edges of $ G_1 $, referred to as white vertices. We then connect two black vertices that share the same name in $ F(G_1) $ if and only if their corresponding vertices in $ G_2 $ are adjacent$^{[3]}$. (See Fig.~2 for an illustration of $ S_3 +_F P_2 $.)
	
	\begin{figure}[htbp]
		\centering
		\includegraphics[height=12.5cm, width=1\textwidth]{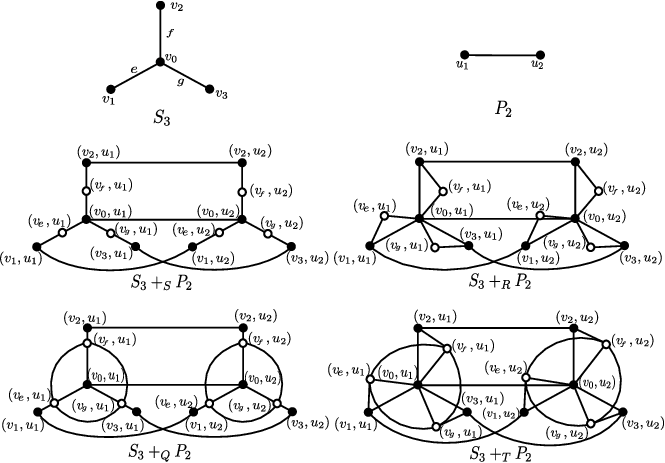}
		
		\centerline{Fig.~2 $S_3 +_F P_2$}
	\end{figure}
	
	In this paper, when we present the vertex ordering on the spine, we assign the color black to the vertices in $ V(G_1) $, and white to the vertices in $ V(F(G_1)) - V(G_1) $, which correspond to the edges of $ G_1 $.

	\begin{lem}$^{[2]}$
		For any simple graph $G$, $\Delta(G)\leq \chi^{\prime }(G) \leq mbt(G),$ where $\chi^{\prime }(G)$ is the chromatic index of $G$.
	\end{lem}
	
	\begin{lem}$^{[2]}$
		For a regular graph \( G \), \( G \) is dispersable only if \( G \) is bipartite.
	\end{lem}
	
	\begin{lem}$^{[5]}$
		An outerplanar graph \( G \) is in class 1 unless \( G \) is an odd cycle.
	\end{lem}

	\begin{lem}$^{[6]}$
		Let \( G = C(\mathbb{Z}_n, \{1, k\}) \), where \( n \) and \( k \) are both even, then \( mbt(G) = \Delta(G) + 1 \).
	\end{lem}	
		
		
	\section {Proof of the main theorem}

	To clarify the vertex relationship between \( F(G) \) and \( G \), we first present the following four lemmas.
	All graphs discussed in this section are simple graphs.
	
	\begin{lem}
		For \( S(G) \) with \( |V(G)| \geq 3 \), we have \( \Delta(S(G)) = \Delta(G) \), and there is always a vertex of maximum degree that is black.
	\end{lem}
	
	\begin{proof}
		The conclusion follows directly from the definition of the \( S \)-operation.
	\end{proof}

	\begin{lem}
		For \( R(G) \), we have \( \Delta(R(G)) = 2\Delta(G) \), and there is always a vertex of maximum degree that is black.
	\end{lem}
	
	\begin{proof}
		The conclusion follows directly from the definition of the \( R \)-operation.
	\end{proof}

	\begin{lem}
		For \( Q(G) \), we have \( \Delta(Q(G)) \geq \Delta(G) + 1 \), and every vertex of maximum degree in \( Q(G) \) is white.
	\end{lem}
	
	\begin{proof}
		Let \( u \) be a vertex of maximum degree in \( G \), and let \( v \) be a neighbor of \( u \). 
		In \( Q(G) \), denote by \( v_e \) the white vertex corresponding to the edge \( uv \). 
		Aside from \( uv \), the vertex \( u \) is incident with \( \Delta(G) - 1 \) other edges, 
		and \( v \) is incident with \( d_G(v) - 1 \) other edges. 
		By the definition of the \( Q \)-operation, the degree of \( v_e \) in \( Q(G) \) is
		\(
		d_{Q(G)}(v_e) = 2 + (\Delta(G) - 1) + (d_G(v) - 1) = \Delta(G) + d_G(v).
		\)
		In particular, when \( d_G(v) = 1 \), we have \( \Delta(Q(G)) \geq \Delta(G) + 1 \). 
		Moreover, by the construction of \( Q(G) \), the degree of the original vertex \( u \) does not increase; 
		therefore, every vertex of maximum degree in \( Q(G) \) must be white.
	\end{proof}
	
	\begin{lem}
		For \( T(G) \), we have \( \Delta(T(G)) = 2\Delta(G) \), and there is always a vertex of maximum degree that is black.
	\end{lem}
	
	\begin{proof}
		We first consider a vertex \( v \) in \( G \). By the definition of the \( T \)-operation, we always have
		\(
		d_{T(G)}(v) = 2d_G(v) \leq 2\Delta(G),
		\)
		and when \( v \) is a vertex of maximum degree, \( d_{T(G)}(v) = 2\Delta(G) \).
		
		We continue to consider an edge \( e \) in \( G \), with endpoints denoted by \( u \) and \( v \). 
		In \( T(G) \), let \( v_e \) be the white vertex corresponding to the edge \( e \). 
		Aside from \( e \), the vertex \( u \) is incident with \( d_G(u) - 1 \) other edges, 
		and \( v \) is incident with \( d_G(v) - 1 \) other edges. 
		By the definition of the \( T \)-operation, the degree of \( v_e \) in \( T(G) \) is
		\(
		d_{T(G)}(v_e) = 2 + (d_G(u) - 1) + (d_G(v) - 1) = d_G(u) + d_G(v) \leq 2\Delta(G).
		\)
		Moreover, if both \( u \) and \( v \) are vertices of maximum degree in \( G \), 
		then \( d_{T(G)}(v_e) = 2\Delta(G) \).
		
		We conclude that
		\(
		\Delta(T(G)) = 2\Delta(G),
		\)
		and there is always a vertex of maximum degree in \( T(G) \) that is black 
		(i.e., corresponds to an original vertex of \( G \)).
	\end{proof}
	
	In~\cite{7}, Liu et al.\ proved that every $2$-connected outerplanar graph which is not an odd cycle is dispersable. In this section, we present a simpler method to prove that every outerplanar graph which is not an odd cycle is dispersable.
	
	\begin{thm}
		If $G$ is an outerplanar graph which is not an odd cycle, then $G$ is dispersable.
	\end{thm}	
	
	\begin{proof}
		\textbf{Case 1:} $\Delta(G) \leq 2$. In this case, $G$ must be either a path or a cycle. Since $G$ is not an odd cycle, it follows that $G$ is dispersable.
		
		\textbf{Case 2:} $\Delta(G) \geq 3$. Since $G$ is an outerplanar graph, all its vertices can be placed on a printing cycle such that all edges lie strictly inside the cycle, and without any crossings. 
		According to Lemma 2.3, $\Delta(G) = \chi^{\prime}(G)$ for $G$ is a class 1 graph. 
		Consequently, all outerplanar graphs with maximum degree at least 3 are dispersable.
	\end{proof}
	
	
	\begin{cor}
		$F \in \{S, Q, R, T\}$, let \( P_n \) be a path with \( n \) vertices, then \( F(P_n) \) is dispersable.
	\end{cor}
	
	\begin{proof}
		Let $P_n$ be a path. 
		Since \( F(P_n) \) is an outerplanar graph which is not an odd cycle, by Theorem 3.1, \( F(P_n) \) is dispersable(see Fig.~3).
	\end{proof}

	\begin{figure}[htbp]
		\centering
		\includegraphics[height=1.7cm, width=0.6\textwidth]{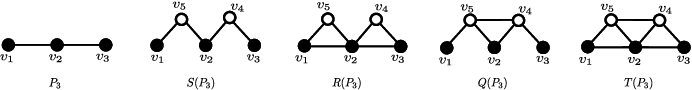}
		\centerline{Fig.~3 $P_3$ and $F(P_3)$}
	\end{figure}
	
	\begin{cor}
		$F \in \{S, Q, R, T\}$, let \( C_n \) be a cycle with \( n \) vertices, then \( S(C_n), Q(C_n),R(C_n) \) is dispersable, and $T(C_n)$ is nearly dispersable.
	\end{cor}
	
	\begin{proof}
		It is straightforward to see that $S(C_n)$, $R(C_n)$, and $Q(C_n)$ are all outerplanar graphs which are not odd cycles. By Theorem~3.1, they are dispersable.
		
		When applying the $T$-operation to $C_n$, we place the white vertex corresponding to each edge of $C_n$ outside the cycle, so that the black and white vertices alternate around a new cycle of length $2n$, denoted $C_{2n}$. All remaining edges are then drawn inside this $C_{2n}$. Consequently, $T(C_n)$ is isomorphic to the circulant graph $C(\mathbb{Z}_{2n}, \{1,2\})$. By Lemma~2.4, it follows that $T(C_n)$ is nearly dispersable (see Fig.~4).
	\end{proof}

	\begin{figure}[htbp]
		\centering
		\includegraphics[height=3.1cm, width=0.6\textwidth]{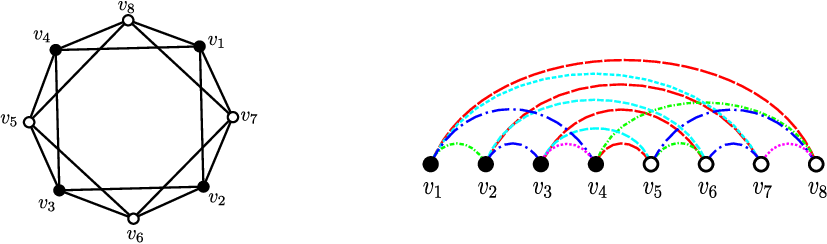}
		\centerline{Fig.~4 $T(C_4)$ and The matching book embedding of $T(C_4)$}
	\end{figure}

	
	\begin{cor}
		$F \in \{S, Q, R, T\}$, Let \( S_n \) is a star graph with $n+1$ vertices\( (n \geq 1) \) , then \( F(S_n) \) is dispersable.
	\end{cor}
	
	\begin{proof}
		The star $S_n$ is a complete bipartite graph $K[1,n]$. Since $S(S_n)$ and $R(S_n)$ are outerplanar graphs which are not odd cycles, Theorem~3.1 implies that they are dispersable. Next, we present a matching book embedding of $Q(S_n)$ and $T(S_n)$.
		
		\noindent \textbf{Case 1: } $Q(S_n)$. 
		
		By Lemma~3.3, we have \( \Delta(Q(S_n)) = \Delta(S_n) + 1 =  n + 1 \).
		By Lemma~2.1, it follows that \( mbt(Q(S_n)) \geq n + 1 \).
		Let $v_0$ denote the central vertex of $Q(S_n)$. Label the vertices in $V(Q(S_n)) - V(S_n)$, which correspond to the white vertices, in counterclockwise order as $v_1, v_2, \ldots, v_n$. Then, starting from the vertex adjacent to $v_n$, label the remaining vertices of $Q(S_n)$ in clockwise order as $v_{n+1}, v_{n+2}, \ldots, v_{2n}$. The ordering along the spine is defined as $v_0, v_1, v_{2n}, v_2, v_{2n-1}, v_3, v_{2n-2}, \ldots,  v_{n-1}, v_{n+2}, v_n, v_{n+1}$.

		\text{Page 1: the edge } $\{v_iv_j \mid i + j = \text{1 } (\text{mod } n + 1) \text{ and } 0 \leq i, j \leq n\}$ and $\{(v_kv_{2n-k+1}) \mid 2k = \text{1 } (\text{mod } n + 1) \text{ and } 1 \leq k \leq n\}$.
		
		\text{Page 2: the edge } $\{v_iv_j \mid i + j = \text{2 } (\text{mod } n + 1) \text{ and } 0 \leq i, j \leq n\}$ and $\{(v_kv_{2n-k+1}) \mid 2k = \text{2 } (\text{mod } n + 1) \text{ and } 1 \leq k \leq n\}$.
		
		$\ldots$ $\ldots$
		
		\text{Page $n$: the edge } $\{v_iv_j \mid i + j = \text{$n$ }  (\text{mod } n + 1) \text{ and } 0 \leq i, j \leq n\}$ and $\{(v_kv_{2n-k+1}) \mid 2k = \text{$n$ }  (\text{mod } n + 1) \text{ and } 1 \leq k \leq n\}$.
		
		\text{Page $n+1$: the edge } $\{v_iv_j \mid i + j = \text{0 } (\text{mod } n + 1) \text{ and } 0 \leq i, j \leq n \}$ and  $\{(v_kv_{2n-k+1}) \mid 2k = \text{0 } (\text{mod } n + 1) \text{ and } 1 \leq k \leq n\}$.
		
		All edges on each page are crossing-free. Consequently, \( Q(S_n) \) admits a matching book embedding in \( n + 1 \) pages. Thus, $Q(S_n)$ is dispersable.(see Fig.~5)
		
		\begin{figure}[htbp]
			\centering
			\includegraphics[height=2.6cm, width=0.63\textwidth]{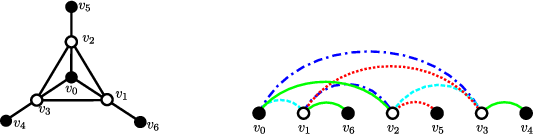}
			\centerline{Fig.~5 $Q(S_3)$ and The matching book embedding of $Q(S_3)$}
		\end{figure}
		
		\noindent \textbf{Case 2: } $T(S_n)$. 
		
		By Lemma~3.3, we have \( \Delta(T(S_n)) = 2n \). 
		By Lemma~2.1, it follows that \( mbt(T(S_n)) \geq 2n \). 
		Since $Q(S_n)$ is a subgraph of $T(S_n)$ and $V(T(S_n)) = V(Q(S_n))$, we adopt the vertex labeling of $Q(S_n)$ from Case~1 for $T(S_n)$. The ordering along the spine is defined as $v_0, v_1, v_{2n}, v_2, v_{2n-1}, \ldots, v_n, $ $v_{n+1}$.

		\text{Page 1: the edge } $\{v_iv_j \mid i + j = \text{1 } (\text{mod } n + 1) \text{ and } 0 \leq i, j \leq n\}$ and $\{(v_kv_{2n-k+1}) \mid 2k = \text{1 } (\text{mod } n + 1) \text{ and } 1 \leq k \leq n\}$.
		
		\text{Page 2: the edge } $\{v_iv_j \mid i + j = \text{2 } (\text{mod } n + 1) \text{ and } 0 \leq i, j \leq n\}$ and $\{(v_kv_{2n-k+1}) \mid 2k = \text{2 } (\text{mod } n + 1) \text{ and } 1 \leq k \leq n\}$.
		
		$\ldots$
		
		\text{Page $n$: the edge } $\{v_iv_j \mid i + j = \text{$n$ } (\text{mod } n + 1) \text{ and } 0 \leq i, j \leq n\}$ and $\{(v_kv_{2n-k+1}) \mid 2k = \text{$n$ }  (\text{mod } n + 1) \text{ and } 1 \leq k \leq n\}$.
		
		\text{Page $n+1$: the edge } $\{v_iv_j \mid i + j = \text{0 } (\text{mod } n + 1) \text{ and } 0 \leq i, j \leq n + 1\}$ and $\{(v_kv_{2n-k+1}) \mid 2k = \text{0 } (\text{mod } n + 1) \text{ and } 1 \leq k \leq n\}$.
		
		\text{Page $n + 2$: the edge } $v_0v_{n+2}$.
		
		$\ldots$ $\ldots$
		
		\text{Page $2n$: the edge } $v_0v_{2n}$.

		This confirms \( T(S_n) \) is dispersable.(see Fig.~6)	
	\end{proof}
	
	\begin{figure}[htbp]
		\centering
		\includegraphics[height=2.6cm, width=0.61\textwidth]{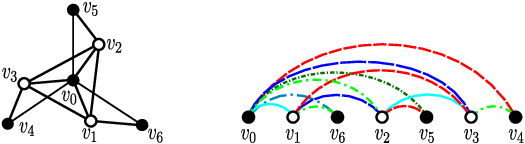}
		\centerline{Fig.~6 $T(S_3)$ and The matching book embedding of $T(S_3)$}
	\end{figure}

	In \cite{8}, Liu et al. proved that for any simple graph $ G $ and any dispersable bipartite graph $ H $, the inequality $ mbt(G \Box H) \leq mbt(G) + \Delta(H) $ holds, where $ \Delta(H) $ denotes the maximum vertex degree of $ H $. In this section, we use a similar approach to prove the following theorem.

	\begin{thm}
		$F \in \{S, Q, R, T\}$, let \( G \) be any simple graph and let \( H \) be any dispersable bipartite graph, then \( mbt(G {+_F} H) \leq mbt(F(G)) + \Delta(H)\), where $\Delta(H)$ is the maximum vertex degree in $H$.
	\end{thm}
	
	\begin{proof}

		Suppose that $ H $ is a dispersable bipartite graph. Since $ H $ is dispersable, there exists a $ \Delta(H) $-edge coloring of $ H $, which allows for a corresponding matching book embedding of $ H $ in a $ \Delta(H) $-page book such that all edges of each color class lie on the same page. Since $ H $ is bipartite, its vertices admit a 2-coloring using blue and red.
		
		Now we embed $ G +_F H $ in the following way. 
		Take a matching book embedding of $ F(G) $ using exactly $ mbt(F(G)) $ pages. 
		Using a dispersable book embedding of $ H $, replace each red vertex of $ H $ with a copy of this matching book embedding of $ F(G) $, and each blue vertex of $ H $ with a reversed copy (i.e., with the vertex ordering reversed). 
		Now we have a copy of $ F(G) $ for each vertex of $ H $. Since each of these copies is placed separately along the spine, the edges of $ G +_F H $ corresponding to the copies of $ F(G) $ can all be placed within $ mbt(F(G)) $ pages.
		
		The rest of the edges of $ G +_F H $ connect corresponding vertices in adjacent (with respect to $ H $) copies of $ F(G) $. 
		Since $ H $ is bipartite, the edges of $ H $ join vertices of different colors. 
		Because the vertex ordering in $ F(G) $ is reversed for the two color classes in $ H $, the set of edges between two adjacent copies of $ F(G) $ corresponding to a single edge of $ H $ can all be drawn on a single page without crossings.
		
		Since the edges of $H$ can be placed on $\Delta(H)$ pages in a matching book embedding, the sets of copies of the edges of $H$ can also be placed on $\Delta(H)$ pages. Hence, all edges of $G{+_F}H$ can be drawn on $mbt(F(G)) + \Delta(H)$ pages.
	\end{proof}
	
	
	For example, let $G = P_3$, $H = C_4$, the matching book embedding of $T(G)$ and $H$ are shown in Fig.~7. 
	The resulting matching book embedding of graph \( G {+_T} H \) using 6 pages is shown in Fig.~8.

	\begin{figure}[htbp]
		\centering
		\includegraphics[height=1.5cm, width=0.5\textwidth]{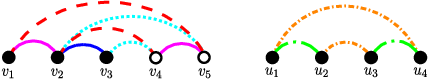}
		\centerline{Fig.~7 The matching book embedding of \( T(P_3) \) (left) and \( C_4 \) (right).}
	\end{figure}
	
	\vspace{1cm}
	
	\begin{figure}[htbp]
		\centering
		\includegraphics[height=4.3cm, width=1\textwidth]{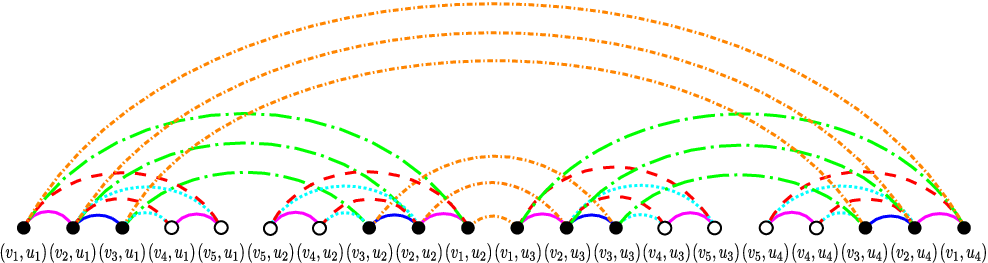}
		\centerline{Fig.~8 The matching book embedding of $P_3 {+_T} C_4$}
	\end{figure}

	\begin{cor}
		$F \in \{S, R, T\}$, Let \( G \) be any simple graph and let \( H \) be dispersable bipartite graph. 
		If \( F(G) \) is dispersable, then \( G +_F H \) is dispersable. where $\Delta(H)$ is the maximum vertex degree in $H$.
	\end{cor}
	
	\begin{proof}
		\textbf{Case 1:} \( |V(G)| = 2 \).  
		For $F \in \{S, R, T\}$, only $S(P_2)$ is dispersable.
		In the results that follow, we prove that $S_1 +_Q H$ is dispersable.
		Since \( P_2 +_S H \) is isomorphic to \( S_1 +_Q H \), it follows that \( P_2 +_S H \) is also dispersable.
		
		\textbf{Case 2:} Let \( F \in \{S, R, T\} \). By Lemmas~3.1, 3.2, and 3.4, \( F(G) \) always contains a vertex of maximum degree that is black. 
		According to the definition of the \( F \)-sum, we have
		\(
		\Delta(G +_F H) = \Delta(F(G)) + \Delta(H).
		\)
		By Lemma~2.1, it follows that
		\(
		mbt(G +_F H) \geq \Delta(F(G)) + \Delta(H).
		\)
		On the other hand, Theorem~3.2 implies
		\(
		mbt(G +_F H) \leq mbt(F(G)) + \Delta(H).
		\)
		Since \( F(G) \) is dispersable, i.e., \( mbt(F(G)) = \Delta(F(G)) \), we conclude that
		\(
		mbt(G +_F H) = \Delta(G +_F H).
		\)
		Hence \( G +_F H \) is dispersable.
	\end{proof}
	
	\begin{rem}
		When $F \in \{S, R, T\}$, Corollary~3.4 shows that there exist graphs $G$ and $H$ such that
		$mbt(G +_F H) = mbt(F(G)) + \Delta(H).$
		For example,
		$mbt(P_3 +_F C_4) = mbt(F(P_3)) + \Delta(C_4).$ Hence, for $F \in \{S, R, T\}$, the upper bound of $mbt(G +_F H)$  given in Theorem~3.2 is tight,
	\end{rem}

	\begin{rem}
		For $G +_Q H$, we have
		$\Delta(G +_Q H) = \max\left\{ \Delta(Q(G)),\, \Delta(G) + \Delta(H) \right\},$
		and
		$\Delta(G +_Q H) < mbt(Q(G)) + \Delta(H).$ In fact, for \( G +_Q H \), Lemma~3.3 implies that every vertex of maximum degree in \( Q(G) \) is white. 
		By the definition of the \( Q \)-sum, we have \( \Delta(G +_{Q} H) = max \{\Delta(Q(G)), \Delta(G) + \Delta(H)\} \). 
		
		When \( \Delta(G) + \Delta(H) \geq \Delta(Q(G))\), we have
		$\Delta(G +_Q H) = \Delta(G) + \Delta(H) < \Delta(Q(G)) + \Delta(H) \leq mbt(Q(G)) + \Delta(H).$
		In particular, if \( G \) is a star graph $S_n$,
		$\Delta(S_n +_Q H) + 1 \leq mbt(Q(S_n)) + \Delta(H)$.
		Otherwise, $\Delta(G +_Q H) + 2 \leq mbt(Q(G)) + \Delta(H).$ 
		
		When \( \Delta(Q(G)) \geq \Delta(G) + \Delta(H) \), we have
		$\Delta(G +_Q H) = \Delta(Q(G)) < \Delta(Q(G)) + \Delta(H) \leq mbt(Q(G)) + \Delta(H).$
		
		The above inequality suggests that it is rather difficult to find suitable graphs $G$ and $H$ satisfying
		$mbt(G +_Q H) = mbt(Q(H)) + \Delta(H).$
	\end{rem}

		The following result shows that
		$mbt(S_n +_Q H) = \Delta(S_n +_Q H) < mbt(Q(S_n)) + \Delta(H).$
	
	
	\begin{thm}       
		Let \( S_n \) be a star graph with $n+1$ vertices\( (n \geq 1) \) and let \( H \) be any dispersable bipartite graph, then \( S_n {+}_Q H \) is dispersable.
	\end{thm}
	
	\begin{proof}
		
		Let the vertex ordering along the spine of a matching book embedding of \( H \) be denoted by \( u_1, u_2, \ldots, u_n \), 
		and let the vertices of \( H \) be colored red and blue.
		Let $E$ be the set of edges on an arbitrarily chosen page of this matching book embedding.
		
		By Lemma~3.3 and the definition of the \( Q \)-sum, we have \( \Delta(S_n {+}_Q H) = \Delta(S_n) + \Delta(H) = n + \Delta(H) \), Lemma 2.1 gives $ mbt(S_n {+}_Q H) \geq n + \Delta(H). $
		
		We first present a matching book embedding of $S_n +_Q H$ for the case $\Delta(H) \geq 2$. We label the vertices of $Q(S_n)$ using the same method as in Corollary~3.3. 
		We replace each red vertex \( u_i \) on the spine of \( H \) with the sequence
		$(v_0, u_i)$, $(v_1, u_i)$, $(v_{2n}, u_i)$,$(v_2, u_i)$, $(v_{2n-1}, u_i)$, $(v_3, u_i)$, $(v_{2n-2}, u_i)$, \ldots, $(v_{n-2}, u_i)$, $(v_{n+3}, u_i)$, $(v_{n-1}, u_i)$, $(v_{n+2}, u_i)$, $(v_n, u_i)$, $(v_{n+1}, u_i)$,
		and replace each blue vertex \( u_j \) on the spine of \( H \) with the sequence
		$(v_{n+1}, u_j)$, $(v_n, u_j)$, $(v_{n+2}, u_j)$, $(v_{n-1}, u_j)$, $(v_{n+3}, u_j)$, $(v_{n-2}, u_j)$, \ldots, $(v_{2n-2}, u_j)$, $(v_3, u_j)$,  $(v_{2n-1}, u_j)$, $(v_2, u_j)$, $(v_{2n}, u_j)$, $(v_1, u_j)$, $(v_0, u_j)$,
		thereby obtaining the vertex ordering of \( S_n +_Q H \) along the spine.

		\text{Page 1: the edge } $\{(v_i,u_h)(v_j,u_h) \mid i + j = \text{1 } (\text{mod } n + 1) \text{ and } 0 \leq i, j \leq n \text{ and } u_h \in V(H)\}$ and $\{(v_k,u_h)(v_{2n-k+1},u_h) \mid 2k = \text{1 } (\text{mod } n + 1) \text{ and } 1 \leq k \leq n \text{ and } u_h \in V(H)\}$ and $\{(v_{i},u_h)(v_{i},u_g) \mid u_hu_g \in E \text{ and } i + 1 = 2n + 1\}$.
		
		\text{Page 2: the edge } $\{(v_i,u_h)(v_j,u_h) \mid i + j = \text{2 } (\text{mod } n + 1) \text{ and } 0 \leq i, j \leq n \text{ and } u_h \in V(H)\}$ and $\{(v_k,u_h)(v_{2n-k+1},u_h) \mid 2k = \text{2 } (\text{mod } n + 1) \text{ and } 1 \leq k \leq n \text{ and } u_h \in V(H)\}$ and $\{(v_{i},u_h)(v_{i},u_g) \mid u_hu_g \in E \text{ and } i + 2 = 2n + 1\}$.
		
		$\ldots$ $\ldots$
		
		\text{Page $n$: the edge } $\{(v_i,u_h)(v_j,u_h) \mid i + j = \text{$n$ } (\text{mod } n + 1) \text{ and } 0 \leq i, j \leq n \text{ and } u_h \in V(H)\}$ and $\{(v_k,u_h)(v_{2n-k+1},u_h) \mid 2k = \text{$n$ } (\text{mod } n + 1) \text{ and } 1 \leq k \leq n \text{ and } u_h \in V(H)\}$ and $\{(v_{i},u_h)(v_{i},u_g) \mid u_hu_g \in E \text{ and } i + n = 2n + 1\}$.
		
		\text{Page $n+1$: the edge } $\{(v_i,u_h)(v_j,u_h) \mid i + j = \text{0 } (\text{mod } n + 1) \text{ and } 0 \leq i, j \leq n \text{ and } u_h \in V(H)\}$ and $\{(v_k,u_h)(v_{2n-k+1},u_h) \mid 2k = \text{0 } (\text{mod } n + 1) \text{ and } 1 \leq k \leq n \text{ and } u_h \in V(H)\}$ and $\{(v_0,u_h)(v_0,u_g) \mid u_hu_g \in E\}$.
		
		Since $H-E$ is still a dispersable bipartite graph, by the  method given in Theorem 3.2, the sets of copies of the edges of $H-E$ can also be placed on $\Delta(H)-1$ pages. 
		Therefore, \( S_n +_Q H \) admits a matching book embedding in \( n + \Delta(H) \) pages.
		Show that $S_n +_{Q}H$ is dispersable.(see Fig.~9)
		
		When $\Delta(H) = 1$, we have $\Delta(S_n +_{Q}H) = n+1$. 
		Let $E = E(H)$, by applying the above method, we obtain an $n+1$ page matching book embedding of \( S_n +_Q H \). 
		Hence, $S_n +_{Q} H$ is dispersable.
	\end{proof}
	
	\begin{figure}[htbp]
		\centering
		\includegraphics[height=2.6cm, width=0.9\textwidth]{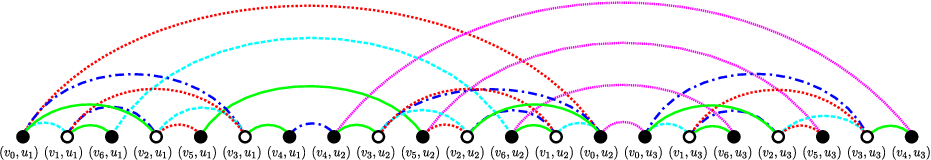}
		\centerline{Fig.~9 The matching book embedding of $S_3 {+_Q} P_3$}
	\end{figure}
	
		
	\begin{thm}
		Let \( P_n \) be a path with \( n \) vertices and let \( H \) be any dispersable bipartite graph, then \( P_n {+_Q} H \) is dispersable.
			
	\end{thm}

	\begin{proof}
		
		We label and color the vertices on the spine of a matching book embedding of \( H \) in the same way as in Theorem~3.3.
		Choose any two pages of this embedding, and denote their edge sets by $E_1$ and $E_2$, respectively. 
		In particular, if $\Delta(H) = 1$, then at least one of $E_1$ or $E_2$ is empty,
		Without loss of generality, we assume $E_2 = \emptyset$.
		
		When \( n \leq 3 \), the path \( P_n \) can be viewed as the star \( S_{n-1} \). 
		Hence, by Theorem~3.3, \( P_n +_Q H \) is dispersable. 
		We now present a matching book embedding of \( P_n +_Q H \) for the case \( n \geq 4 \).
		 
		We first describe a matching book embedding of $P_n +_Q H$ for the case $\Delta(H) \geq 3$. 
		By Lemma~3.3 and the definition of the \( Q \)-sum, we have $\Delta(Q(P_n)) = 4$ and $\Delta(P_n +_Q H) = \Delta(H) + 2$. 
		By Lemma~2.1, it follows that $mbt(P_n +_Q H) \geq 2 + \Delta(H)$. 
		In \( Q(P_n) \), we label the black vertices from left to right as \( v_1, v_2, v_3, \dots, v_n \), 
		and label the white vertices in the reverse order as \( v_{n+1}, v_{n+2}, \dots, v_{2n-1} \).(see Fig.~10)
		
		\begin{figure}[htbp]
			\centering
			\includegraphics[height=1.5cm, width=0.25\textwidth]{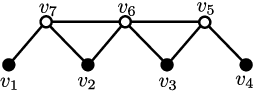}
			\centerline{Fig.~10 $Q(P_4)$}
		\end{figure}
		
		We replace each red vertex \( u_i \) on the spine of \( H \) with the sequence
		$(v_1, u_i)$, $(v_2, u_i)$, $(v_3, u_i)$, \ldots, $(v_{n-1}, u_i)$, $(v_n, u_i)$,
		and replace each blue vertex \( u_j \) on the spine of \( H \) with the reversed sequence
		$(v_n, u_j)$, $(v_{n-1}, u_j)$, \ldots, $(v_3, u_j)$, $(v_2, u_j)$, $(v_1, u_j)$,
		thereby obtaining the vertex ordering of \( P_n +_Q H \) along the spine.
		
		\text{Page 1: the edge } $\{(v_i,u_h)(v_j,u_h) \mid i + j = 2n \text{ and } 1 \leq i, j \leq 2n - 1 \text{ and } u_h \in V(H)\}$.
		
		\text{Page 2: the edge } $\{(v_i,u_h)(v_j,u_h) \mid i + j = 2n + 1 \text{ and } 1 \leq i, j \leq 2n - 1 \text{ and } u_h \in V(H)\}$.
		
		\text{Page 3: the edge } $\{(v_i,u_h)(v_{i+1},u_h) \mid i \geq n+1 \text{ and } i \text{ is an odd}\}$ and $\{(v_i,u_h)(v_i,u_g) \mid 1 \leq i \leq n \text{ and } u_hu_g \in E_1\}$.
		
		\text{Page 4: the edge } $\{(v_i,u_h)(v_{i+1},u_h) \mid i \geq n+1 \text{ and } i \text{ is an even}\}$ and $\{(v_i,u_1)(v_i,u_2) \mid 1 \leq i \leq n \text{ and } u_hu_g \in E_2\}$.
		
		Since $H-(E_1 \cup E_2)$ is still a dispersable bipartite graph, by the method given in Theorem 3.1, the sets of copies of the edges of $H-(E_1 \cup E_2)$ can also be placed on $\Delta(H)-2$ pages. 
		Therefore, \( P_n +_Q H \) admits a matching book embedding in \( \Delta(H) + 2 \) pages.
		Show that $P_n +_{Q} H$ is dispersable.(see Fig.~11)
		
		When $\Delta(H) = 1$, we have $\Delta(P_n +_{Q} H) = 4$. 
		Let $E_1 = E(H)$ and $E_2 = \emptyset$, by applying the above method, we obtain a 4-page matching book embedding of \( P_n +_Q H \). 
		Hence, $P_n +_{Q} H$ is dispersable.
		
		When $\Delta(H) = 2$, we have $\Delta(P_n +_Q H) = 4$. 
		Let $E_1 \cup E_2 = E(H)$, by applying the above method, we obtain a 4-page matching book embedding of \( P_n +_Q H \). 
		Hence, $P_n +_{Q} H$ is dispersable.
	\end{proof}
		
		\begin{figure}[htbp]
			\centering
			\includegraphics[height=2.9cm, width=1\textwidth]{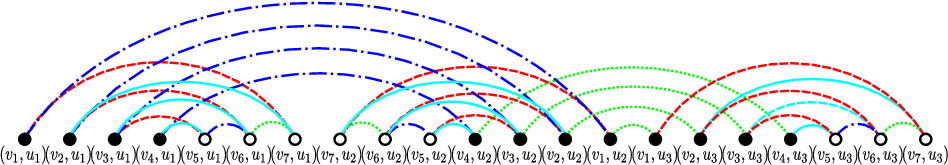}
			\centerline{Fig.~11 The matching book embedding of $P_4 {+_Q} P_3$}
		\end{figure}

	Theorem~3.3 and Theorem~3.4 show that for paths and stars, we have
	\(
	mbt(G +_Q H) \leq \Delta(G) + \Delta(H).
	\)
	However, this inequality does not hold for the \( Q \)-sum of all simple graphs \( G \) and dispersable bipartite graphs \( H \).
	Indeed, the following result shows that
	\(
	mbt(C_p +_Q C_q) = 5 > \Delta(C_p) + \Delta(C_q),
	\)
	where \( q \) is even.
	
	\begin{thm}       
		\( C_p +_Q C_q \) (with \( q \) even) is nearly dispersable.
	\end{thm}
	
	\begin{proof}
		By Lemma~3.3, every white vertex in \( Q(C_p) \) has degree 4, and every black vertex has degree 2. 
		According to the definition of the \( Q \)-sum, it follows that \( C_p +_Q C_q \) is a 4-regular graph containing an odd cycle. 
		Since bipartite graphs contain no odd cycles, Lemma~2.2 implies that
		\(
				mbt(C_p +_Q C_q) \geq 5.
		\)
		We now present a matching book embedding of \( C_p +_Q C_q \).
		
		Let the vertex ordering along the spine of a matching book embedding of \( C_q \) be denoted by \( u_1, u_2, \ldots, u_n \), 
		and let the vertices of \( C_q \) be colored red and blue.
		Since \( mbt(C_q) = 2 \), let \( E_1 \) and \( E_2 \) denote the edge sets of the two pages in a matching book embedding of \( C_q \). Then
		\(
		E_1 \cup E_2 = E(C_q) \text{ and } E_1 \cap E_2 = \emptyset.
		\)
		
		In \( Q(C_p) \), we label the black vertices clockwise as \( v_1, v_2, v_3, \dots, v_p \). 
		Starting from the white vertex corresponding to the edge \( v_p v_{p-1} \), 
		we label the white vertices counterclockwise as \( v_{p+1}, v_{p+2}, \dots$ $, v_{2p} \). (see Fig.~12)
		
		We replace each red vertex \( u_i \) on the spine of \( C_q \) with the sequence
		$(v_{1}, u_{i})$, $(v_{2}, u_{i})$, $(v_{3}, u_{i})$, $\cdots$, $(v_{p-1}, u_{i})$, $(v_{p}, u_{i})$, 
		and replace each blue vertex \( u_j \) on the spine of \( C_q \) with the sequence
		$(v_{p}, u_{j})$, $(v_{p-1}, u_{j})$, $(v_{p-2}, u_{j})$ $\cdots$, $(v_{2}, u_{j})$, $(v_{1}, u_{j})$, 
		thereby obtaining the vertex ordering of \( C_p +_Q C_q \) along the spine.

		(1) $p$ is even
		
		\text{Page 1: the edge } $\{(v_i,u_h)(v_j,u_h) \mid i + j = \text{0 } (\text{mod } 2p) \text{ and } 0 \leq i, j \leq 2p \text{ and } u_h \in V(C_q)\}$ 
		
		\text{Page 2: the edge } $\{(v_i,u_h)(v_j,u_h) \mid i + j = \text{1 } (\text{mod } 2p) \text{ and } 0 \leq i, j \leq 2p \text{ and } u_h \in V(C_q)\}$ 
		
		\text{Page 3: the edge } $\{(v_p,u_h)(v_{2p},u_h) \mid u_h \in V(C_q)\}$ and $\{(v_i,u_h)(v_{i+1},u_h) \mid p+1 \leq i \leq 2p-3 \text{ and } i \text{ is } odd \text{ and } u_h \in V(C_q)\}$.
		
		\text{Page 4: the edge } $\{(v_{p+1},u_h)(v_{2p},u_h) \mid u_h \in V(C_q)\}$ and $\{(v_i,u_h)(v_{i+1},u_h) \mid p+2 \leq i \leq 2p-2 \text{ and } i \text{ is } even \text{ and } u_h \in V(C_q)\}$ and  $\{(v_{i},u_h)(v_{i},u_g) \mid 1 \leq i \leq p \text{ and } u_hu_g \in E_1\}$.
		
		\text{Page 5: the edge } $\{(v_{2p-1},u_h)(v_{2p},u_h) \mid u_h \in V(C_q)\}$ and  $\{(v_{i},u_h)(v_{i},u_g) \mid 1 \leq i \leq p \text{ and } u_hu_g \in E_2\}$.
		
		(2) $p$ is odd
		
		\text{Page 1: the edge } $\{(v_i,u_h)(v_j,u_h) \mid i + j = \text{0 } (\text{mod } 2p) \text{ and } 0 \leq i, j \leq 2p \text{ and } u_h \in V(C_q)\}$ 
		
		\text{Page 2: the edge } $\{(v_i,u_h)(v_j,u_h) \mid i + j = \text{1 } (\text{mod } 2p) \text{ and } 0 \leq i, j \leq 2p \text{ and } u_h \in V(C_q)\}$ 
		
		\text{Page 3: the edge } $\{(v_p,u_h)(v_{2p},u_h) \mid u_h \in V(C_q)\}$ and $\{(v_i,u_h)(v_{i+1},u_h) \mid p+1 \leq i \leq 2p-2 \text{ and } i \text{ is } even \text{ and } u_h \in V(C_q)\}$.
		
		\text{Page 4: the edge } $\{(v_{p+1},u_h)(v_{2p},u_h) \mid u_h \in V(C_q)\}$ and $\{(v_i,u_h)(v_{i+1},u_h) \mid p+2 \leq i \leq 2p-3 \text{ and } i \text{ is } odd \text{ and } u_h \in V(C_q)\}$ and  $\{(v_{i},u_h)(v_{i},u_g) \mid 1 \leq i \leq p \text{ and } u_hu_g \in E_1\}$.
		
		\text{Page 5: the edge } $\{(v_{2p-1},u_h)(v_{2p},u_h) \mid u_h \in V(C_q)\}$ and  $\{(v_{i},u_h)(v_{i},u_g) \mid 1 \leq i \leq p \text{ and } u_hu_g \in E_2\}$.
		
		Therefore, \( C_p +_Q C_q \) is nearly dispersable.(see Fig.~13)
	\end{proof}		
	
	\begin{figure}[htbp]
		\centering
		\includegraphics[height=2cm, width=0.4\textwidth]{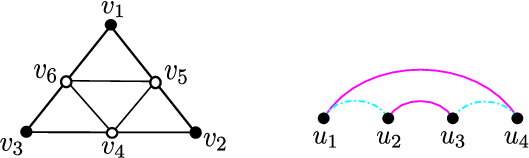}
		\centerline{Fig.~12  $Q(C3)\text{ and }C4$}
	\end{figure}
	
	\begin{figure}[htbp]
		\centering
		\includegraphics[height=3.7cm, width=1\textwidth]{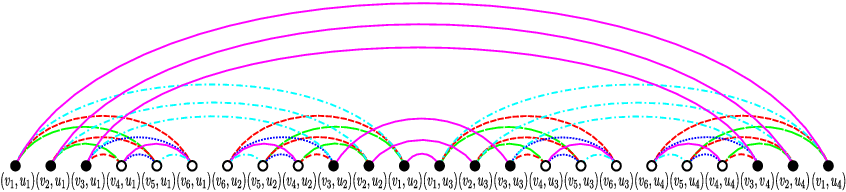}
		\centerline{Fig.~13  The matching book embedding of $C_3 +_Q C_4$}
	\end{figure}

	\section{Conclusions}
		
		In this paper, we show that outerplanar(except odd cycles) graphs is dispersable. 
		Additionally, we establish an upper bound on the matching book thickness of the $F$-sum of any simple graph 
		and any dispersable bipartite graph. 
		Meanwhile, for $F \in \{S, R, T\}$, the upper bound provided in Theorem~3.2 is tight, we investigate the relationship between $\Delta(G +_ Q H)$ and $mbt(Q(G)) + \Delta(H)$. However, we have not yet found suitable graphs $G$ and $H$ such that
		$mbt(G +_Q H) = mbt(Q(G)) + \Delta(H).$
		Moreover, we observe that $C_p +_Q C_q$ (with $q$ odd) is a $4$-regular graph  containing an odd cycle. This naturally leads us to propose the following question.
		
		\noindent\textbf{Question:} 
		\begin{enumerate}
			\item Is \( C_p +_Q C_q \) (with $q$ odd) nearly dispersable?
			\item For the graph \( G +_Q H \), is the upper bound on the matching book thickness given in Theorem~3.2 tight?
		\end{enumerate}
		
		\section*{Acknowledgment}
			We would like to thank the anonymous referee for very helpful comments and suggestions. 	
			This work was supported by the National Natural Science Foundation of China (No.~12571345).

	\end{document}